\documentclass[12pt]{amsart}
\usepackage{amssymb,amscd,amsmath,amsthm,amsfonts}
\usepackage{amsthm}
\usepackage{fullpage}
\usepackage[all]{xy}

\def\ad{{\rm ad}}
\def\alg{{\mathcal A}}
\def\cpt{{\mathcal K}}
\def\torus{{\bf T}}
\def\wh#1{\widehat{#1}}
\def\torus{{\bf T}}
\def\integer{{\bf Z}}

\def\cH{{\mathcal H}}
\def\alg{{\mathcal A}}
\def\bdd{{\mathcal B}}
\def\cpt{{\mathcal K}}

\def\CT{{\rm CT}}

\def\wt#1{\widetilde{#1}}

\newcommand{\what}{\widehat}

\newcommand{\RR}{{\bf R}}
\newcommand{\PUH}{{\sf PU}({\mathcal H})}
\newcommand{\TT}{{\bf T}}

\newcommand{\ZZ}{{\bf Z}}

\newcommand{\cK}{{\mathcal K}}

\newcommand{\thm}[1]{\advance\count32 by 1\bigskip\noindent{\bf #1 \the\count31.\the\count32.}}

\newcommand{\quod}{\hfill\qed\bigskip}%

\theoremstyle{plain}
\newtheorem{theorem}{Theorem}[section]

\newtheorem{proposition}[theorem]{Proposition}
\newtheorem{corollary}[theorem]{Corollary}

\theoremstyle{definition}
\newtheorem{definition}[theorem]{Definition}
\newtheorem{definition*}{Definition}
\newtheorem{example}[theorem]{Example}
\theoremstyle{remark}

\newtheorem{remarks*}{Remarks}

\begin{document}

\title[Nonassociative strict deformation quantization of $C^*$-algebras]
{Nonassociative strict deformation quantization of $C^*$-algebras and
nonassociative torus bundles}

\author[KC Hannabuss]{Keith C. Hannabuss}

\address[Keith Hannabuss]{
Mathematical Institute \\
24-29 St. Giles' \\
Oxford, OX1 3LB\\
and Balliol College \\
Oxford, OX1 3BJ\\
England} \email{kch@balliol.ox.ac.uk}

\author[V Mathai]{Varghese Mathai}

\address[Varghese Mathai]{
Department of Pure Mathematics\\
University of Adelaide\\
Adelaide, SA 5005\\
Australia}
\email{mathai.varghese@adelaide.edu.au}

\begin{abstract}
In this paper, we initiate the study of nonassociative strict deformation
quantization of $C^*$-algebras with a torus action. We shall also present a
definition of nonassociative principal torus bundles, and give a classification
of these as nonassociative strict deformation quantization of ordinary
principal torus bundles. We then relate this to T-duality of principal torus
bundles with $H$-flux. In particular, the Octonions fit nicely into our theory.
\end{abstract}

\thanks{This research was supported under Australian Research Council's Discovery
Projects funding scheme (project number DP110100072).
Both authors
thank Peter Bouwknegt for discussions during previous collaborations}

\keywords{Nonassociative strict deformation quantization, nonassociative torus, 
nonassociative torus bundles, T-duality}

\subjclass[2010]{Primary 46L70 Secondary:  47L70, 46M15,18D10,  46L08,
46L55} 

\maketitle


\section*{Introduction}


We shall present a definition of nonassociative principal torus
bundles, inspired in part by the work on noncommutative
principal torus bundles \cite{ENOO, HM09, HM10}, and by our
earlier work with Bouwknegt, \cite{BHM3, BHM4}, on
nonassociative $C^*$-algebras that were T-duals of spacetimes that are principal torus bundles with background H-flux.
In the previous work (with Bouwknegt), we were unable to give an independent geometric description
of these nonassociative algebras, 
however here, we achieve such a description in many interesting
special cases via nonassociative strict deformation quantization initiated in this paper.
More detailed motivation for this paper arising from string theory is explained in Section 2.

Associative operator theoretic deformation quantization was studied in detail
by Rieffel \cite{Rieffel1} (see the references therein), who called it
{\em strict} deformation quantization, mainly to distinguish it from {\em formal} deformation
quantization. Formal nonassociative deformation quantization has been previously studied
in the mathematical physics literature, cf. \cite{CS, HKK, NS, Paal}, but as far as we know, ours
 is the first paper to analyse {\em strict} nonassociative deformation quantization via
nonassociative operator algebras or $C^*$-algebras in tensor categories \cite{BHM3, BHM4}.
More precisely, we define and study nonassociative strict deformation
quantization of $C^*$-algebras with a torus action, and apply it to show that
some of the T-duals of spacetimes with H-flux that are principal torus bundles (of rank $\ge 3$),
are nonassociative strict deformation
quantization of principal torus bundles (with basic H-flux), thereby giving a satisfactory
description of the T-dual at least in these cases.

In \cite{BHM3, BHM4} we showed how a natural definition of
T-duals in terms of crossed products could lead to
nonassociative algebras. For a locally compact group $G$ and a
 Moore 3-cocycle  $\phi \in Z^3(G,\torus)$ (or more conveniently
 an antisymmetric tricharacter \cite{BHM3}), the twisted compact operators
 $\cpt_\phi(L^2(G))$ provide a simple example of such an
 algebra: one starts by giving the integral kernels $K_j$ defining compact
 operators a  deformed product
$$
(K_1\star K_2)(x,z) = \int_G \phi(x,y,z)K_1(x,y)K_2(y,z)\,dy.
$$
More generally, nonassociativity can occur for an algebra
$\alg$ whenever $C_0(G)$ acts on $\alg$, so that $\phi\in
C_0(G\times G\times G)$ acts on $(\alg\otimes\alg)\otimes
\alg$. Using this to identify $(\alg\otimes\alg)\otimes \alg$
with $\alg\otimes(\alg\otimes \alg)$ leads to nonassociativity
in products of three elements in $\alg$. (We shall refer to
this as $\phi$-nonassociativity.) In the case of the kernels we
can take the action of $f\in C_0(G)$ to be $(f.K)(z,w) =
f(zw^{-1})K(z,w)$, \cite[Section 10]{BHM3}. An action $\beta$
(of $G$) by automorphisms of $\alg$ is consistent with this if
$\beta_x[f.a] = \rho_x[f].\beta_x[a]$ where $\rho_x[f](y) =
f(xy)$.

One of the aims of this paper is to show that the
nonassociative bundles of \cite{BHM3} have a characterisation
analogous to that of the noncommutative principal bundles of
\cite{ENOO,HM10}. In fact, we shall see that this automatically
allows for noncommutativity as well.

The paper starts with a section showing that nonassociative
torus algebras can be viewed as nonassociative strict
deformation quantizations of their associative counterparts,
and Section 2 relates these nonassociative torus algebras to T-duality, which is a fundamental
symmetry in string theory. In Section 3 we
introduce a definition of nonassociative principal torus
bundles and show in the next Section 4 that the bundles
encountered in \cite{BHM3} are examples of such bundles. This
involves an interesting new variant of the Takai duality
Theorem, which combines duality with MacLane's notion of
strictification to strictly associative categories. The
nonassociative principal torus bundles are then classified in
Section 5 by a minor modification of the classification of
noncommutative principal torus bundles \cite{ENOO,HM10}.
Finally, in the appendix we show that the Octonions (stabilized) can also be recovered as a
nonassociative strict deformation quantization.


\section{Nonassociative torus and nonassociative strict deformation quantization}
\label{sect:natorus}


In this section, given a $T-C^*$-algebra $\alg$, where $T$ is a torus,  together with a circle valued 3-cocycle $\phi$
on $\what T$, we will define the  {nonassociative strict deformation quantization}
$\alg_\phi$ of $\alg$, which is a $C^*$-algebra in a tensor category with associator equal to $\phi$, see \cite{BHM4}. 
The nonassociative torus $T_\phi$, introduced in \cite{BHM3}, is an example of this construction.

Let $T=V/\Lambda$ be a torus, written as the quotient of a vector group $V$ and a maximal rank lattice $\Lambda$, and $\what T$ denote the Pontryagin dual of
$T$. Let $\phi \in Z^3(\what T, \TT)$ be a $\TT$-valued
3-cocycle on $\what T$.

Following the constructions of \cite[Section 4]{BHM3}, consider
the unitary operator $u(\beta, \gamma)$ acting on $L^2(\what T)$
given by
\begin{equation*}
(u(\beta, \gamma) \psi)(\alpha) = \phi(\alpha, \beta, \gamma) \psi(\alpha)
\end{equation*}
for all $\psi \in L^2(\what T)$. We easily verify that
\begin{equation*}
\phi(\alpha, \beta, \gamma)  u(\alpha, \beta) u(\alpha\beta, \gamma)  =
\xi_\alpha[u(\beta, \gamma)]  u(\alpha, \beta\gamma)
\end{equation*}
where $\xi_\alpha = {\rm ad}(\rho(\alpha))$ and
$(\rho(\alpha)\psi)(g) = \psi(g\alpha)$ is the right regular
representation. One can define, as was done in \cite{BHM3,
BHM4}, a twisted convolution product and adjoint on $C(\what T,
\cK)$, where $\cK = \cK(L^2(\what T))$,  by
\begin{equation}
(f*_\phi g)(x) = \sum_{y\in \Lambda} f(y) \xi_y[g(y^{-1}x)] u(y, y^{-1}x)  \,,
\end{equation}
and
\begin{equation}
f^*(x) = u(x, x^{-1})^{-1} \xi_x[f(x^{-1})]^* \,.
\end{equation}
The operator norm completion is the nonassociative twisted
crossed product C$^*$-algebra known as the {\em nonassociative
torus} which was first defined in \cite{BHM3},
\begin{equation*}
 T_\phi= \cK(L^2(\what T)) \rtimes_{\xi, u} \what T.
\end{equation*}
When $\phi$ is trivial, the nonassociative torus specializes to
the continuous functions on the torus $T$, stabilized.

Our next goal is to extend this construction to general
$C^*$-algebras with a $T$-action. Let $\alg$ be a $C^*$-algebra
with a continuous action $\alpha$ of $T$ and $\phi \in
Z^3(\what T, \TT)$ be a $\TT$-valued 3-cocycle on $\what T$.
Then we define the {\em nonassociative strict deformation
quantization} of $\alg$, denoted $\alg_\phi$ as follows.

We have the direct sum decomposition,
$$
\begin{array}{lcl}
\alg \otimes \cK&\cong & \widehat\bigoplus_{\chi \in \what T} \alg_\chi \otimes \cK
 \\
\qquad {\bf a} &=& \sum_{\chi \in \what T} {\bf a_\chi},
\end{array}
$$
where $\cK = \cK(L^2(\what T))$ and for $\chi\in\widehat{T}$,
$$
\alg_\chi:= \left\{a\in \alg \mid
  \alpha_t(a)=\chi(t)\cdot a\,\,
 \,\,\,   \forall\, t\in T \right\}.
$$

Since $T$ acts by $\star$-automorphisms, we have
\begin{equation*}
  \label{eq:Fell_bundle_algebra}
  \alg_\chi\cdot \alg_\eta \subseteq \alg_{\chi\eta}
  \quad {\rm and} \quad \alg_\chi^*=\alg_{\chi^{-1}}
  \qquad  \forall\, \chi,\eta\in\widehat{T}.
\end{equation*}

The completion of the direct sum is explained as follows.\\
The representation theory of $T$ shows that
\(\bigoplus_{\chi\in\widehat{T}} \alg_\chi\) is a
$T$-equivariant dense subspace of~\(\alg\), where $T$ acts on $
\alg_\chi$ as follows: $\what \alpha_t(a_\chi) = \chi(t) a_\chi$
for all $t\in T$. Then $\widehat\bigoplus_{\chi \in \what T}
\alg_\chi$ is the completion in the $C^*$-norm of $\alg$.
Let
$$
(\alg\otimes\cK)^{alg} = \bigoplus_{\chi \in \what T} \alg_\chi \otimes \cK .
$$

The product on  $\alg\otimes\cK$  then also decomposes as,
$$
({\bf a}{\bf  b})_\chi= \sum_{\chi_1\chi_2=\chi} {\bf a}_{\chi_1} {\bf b}_{\chi_2}
$$
for $\chi_1,\chi_2, \chi \in \widehat T$. The product can be
deformed to a nonassociative product $\star_\phi$ on
$\alg\otimes\cK$ by setting
$$
({\bf a}\star_\phi {\bf b})_\chi = \sum_{\chi_1\chi_2=\chi} {\bf a}_{\chi_1}
\xi_{\chi_1}[{\bf b}_{\chi_2}] u(\chi_1, \chi_2).
$$

We next describe the norm completion of the deformed algebra.
Let $\cH_1$ denote the universal Hilbert space representation of the
$C^*$-algebra $\alg\otimes\cK$ (with its usual product) which one obtains via the GNS theorem. By considering instead the Hilbert space $\cH=\cH_1 \otimes L^2(T) \otimes \cH_2$, where $\cH_2$ is an infinite dimensional Hilbert space, where we note that every character of $T$ occurs with infinite multiplicity in $L^2(T) \otimes \cH_2$, we obtain a $T$-equivariant embedding $\varpi:\alg\otimes\cK \to B(\cH)$. The equivariance means that
$$
\varpi({\bf a}_\chi) = \varpi({\bf a})_\chi.
$$
Now consider the action of $(\alg\otimes\cK)^{alg}$ on $\cH$ given by the deformed product $\star_\phi$, 
that is, for ${\bf a} \in (\alg\otimes\cK)^{alg}$ and $\Psi \in \cH$,
$$
({\bf a}\star_\phi \Psi)_\chi = \sum_{\chi_1\chi_2=\chi} \varpi({\bf a}_{{\chi}_1})\xi_{\chi_1}[\Psi_{\chi_2}]u(\chi_1, \chi_2).
$$
The operator norm completion of this action is by definition the  {\em nonassociative strict deformation quantization}
$\alg_\phi$ of $\alg$. It is a $C^*$-algebra in a tensor category with associator equal to $\phi$, see \cite{BHM4}.
Let $\alg= C(T)$, which is a $T-C^*$-algebra where $T$ acts on itself by translation. Let $\phi$ be a circle valued 
3-cocycle on $\what T$. Then $C(T)_\phi \cong C^*(\what T)_\phi$, is just the {\em nonassociative torus}, also denoted by $T_\phi$, which was previously described. 


We might ask whether this can be understood using the Landstad--Kasprzak approach, as in \cite{HM10}. (A crossed product $\bdd= \alg\rtimes_\alpha G$ of a C$^*$ algebra $\alg$ by a locally compact abelian group $G$ acting as automorphisms by $\alpha:G \to
{\rm Aut}(\alg)$ has a dual action $\wh{\alpha}$ of the dual group $\wh{G}$. Landstad noticed that the crossed product also has a coaction of $G$ given by a homomorphism $\lambda: G \to {\mathcal UMB}$ to the unitaries in the multiplier algebra, such that $\wh{\alpha}(\xi)(\lambda(g)) = \xi(g)\lambda(g)$ for all $g\in G$ and $\xi \in \wh{G}$; moreover, the existence of such related $\wh{G}$- action and $G$-coaction on a C$^*$-algebra $\bdd$, characterises $\bdd$ as a crossed product. Kasprzak noticed that for such an algebra $\bdd$ and an antisymmetric bicharacter $u$ defining a 2-cocycle on $\wh{G}$ one can obtain a deformed action $\wh{\alpha}^u(\xi) = \lambda(u(\xi,\cdot))\wh{\alpha}(\xi)$, also satisfying Landstad's conditions. Consequently there must exist a deformed algebra $\alg_u$, such that $\alg_u\rtimes_{\alpha_u} G \cong \bdd \cong \alg\rtimes_{\alpha} G$, where $\alpha_u$ is the dual to $\wh{\alpha}^u$.) However, if we followed precisely the route in \cite{BHM3}, then a nonassociative algebra $\alg^u$ would have a nonassociative crossed product $\bdd$ and we could not possibly have $\alg^u\rtimes_{\alpha^u}G \cong \bdd \cong \alg\rtimes_{\alpha}G$. Of course, the constructions of \cite[Appendix A]{BHM3}, which we discuss below, allow us to strictify the product in the sense of category theory, that is to remove the nonassociativity. It does seem likely that a more sophisticated approach along those lines could work, but we leave that for future investigation.

We should also note that this notion of deformation is not
quite the usual one since the condition that the restriction of
$\phi$ to $\Lambda^\perp \times\Lambda^\perp
\times\Lambda^\perp$ is trivial builds in a certain rigidity (where $\Lambda^\perp \subseteq \what{V}$ is the set of characters trivial on $\Lambda$).


\section{The relation to T-duality}


We begin by summarizing T-duality for principal torus bundles with H-flux.
T-duality, also known as target space duality, plays an important role in string theory
and has been the subject of intense study for many years.
In its most basic form, T-duality relates a string theory compactified
on a circle of radius $R$, to a string theory compactified on the dual circle of radius $1/R$
by the interchange of the string momentum and winding numbers.
T-duality can be generalized to locally defined circles (principal circle bundles, circle fibrations),
higher rank torus bundles or fibrations, and, in the presence of a background H-flux which is
represented by a closed, integral \v Cech 3-cocycle $H$ on the spacetime manifold
$Y$, it is closely related to mirror symmetry
through the SYZ-conjecture, \cite{SYZ}.

A striking feature of T-duality is that it can relate topologically distinct spacetime manifolds by the
interchange of topological characteristic classes with components of the H-flux.
Specifically, denoting by $(Y,[H])$ the pair of an (isomorphism class of) principal
circle bundle $\pi: Y\to X$, characterized by the first Chern class $[F]\in H^2(X,\ZZ)$
of its associated line bundle,  and an H-flux $[H] \in H^3(Y,\ZZ)$, the T-dual again
turns out to be a pair $(\widehat Y,[\widehat H])$, where the principal circle bundles

$$\xymatrix{
\TT  \ar[r] &  Y \ar[d]_{\pi}  \\
& X}, \qquad \qquad \qquad \qquad
\xymatrix{
\TT  \ar[r] &  \widehat Y \ar[d]_{\widehat\pi}  \\
& X}
$$
are related by
$[\widehat F] = \pi_*[H] ,\  [F] = \widehat \pi_* [\widehat H]$,
such that on the correspondence space
$$
\xymatrix @=4pc @ur { Y \ar[d]_{\pi} &
Y\times_X  \widehat Y \ar[d]_{\widehat p} \ar[l]^{p} \\ X & \widehat Y\ar[l]^{\widehat \pi}}
$$
we have $p^* [H] - \widehat p^*[\widehat H] = 0$ \cite{BEM03a,BEM03b}.

In earlier papers we have argued that the twisted K-theory $K^\bullet (Y,[H])$
(see, e.g., \cite{BCMMS01})
classifies charges of D-branes on $Y$ in the background of H-flux $[H]$ \cite{BM00},
and indeed, as a consistency check, one can prove that T-duality gives an isomorphism
of twisted K-theory (and the closely related twisted cohomology $H^\bullet(Y,[H])$ by
means of the twisted Chern character $ch_H$ depicted by the vertical maps in the diagram below) 
\cite{BEM03a}
$$\xymatrix{
K^\bullet(Y, [H])  \ar[r]^{T_!} \ar[d]^{ch_H} &
      K^{\bullet+1}(\widehat Y, [ \widehat H]) \ar[d]^{ch_{\widehat H} } \\
H^\bullet  (Y, [H])    \ar[r]^{T_*} &   H^{\bullet+1} (\widehat Y, [\widehat H])
}
$$
where the top horizontal map is the T-duality isomorphism in (twisted) K-theory and 
the bottom horizontal map is the T-duality isomorphism in (twisted) de Rham cohomology.
The above considerations were generalized to principal torus bundles in
\cite{BHM3,BHM4}.

Since the projective unitary group of an infinite dimensional Hilbert space
$\PUH$ is a model for $K(\ZZ,2)$, we can `geometrize' the H-flux in terms of (an
isomorphism class of) a principal $\PUH$-bundle $P$ over $Y$.  We can reformulate the
discussion of T-duality above in terms of noncommutative geometry as follows.
The space of continuous sections vanishing at infinity, $\alg = C_0(Y,\mathcal E)$,
of the associated algebra bundle of compact operator $\cK$ on the Hilbert space
$\mathcal E = P \times_{\PUH} \mathcal K$, is a stable, continuous-trace, C$^*$-algebra
with spectrum $Y$, and has the property that it is locally Morita equivalent to continuous
functions on $Y$. Thus the $H$-flux has the effect of making spacetime noncommutative.
The K-theory of $\alg$ is just the twisted K-theory $K^\bullet (Y,[H])$.
The $\TT$-action on $Y$ lifts essentially uniquely to an $\RR$-action
on $\alg$.  In this context T-duality is the operation of taking the
crossed product $\alg \rtimes \RR$, which turns out to be another continuous
trace algebra associated to $(\widehat Y, [\widehat H])$ as above. A fundamental property of T-duality is that
when it is applied twice, it yields the original algebra $\alg$, and the reason that it works in this
case is due to Takai duality.   The isomorphism of the D-brane charges
in twisted K-theory is, in this context, due to the Connes-Thom isomorphism.
These methods have been generalized to principal torus bundles
in \cite{MR04a, MR04b, MR05}, however novel features arise.
First of all the $\TT^n$-action on the principal
torus bundle $Y$ need not always lift to an $\RR^n$-action on
$\alg$.  Even if
it does, this lift need not be unique.  Secondly, the crossed product $\alg \rtimes \RR^n$
need not be a continuous-trace algebra, but rather, it could be a continuous field of
noncommutative tori
\cite{MR04a}, and necessary and sufficient conditions are given for when these T-duals occur.
In recent work \cite{HM09, HM10}, we describe the T-duals of trivial torus bundles with
H-flux that has no component $H_0$ as below, in terms of parametrised strict deformation
quantization of principal torus bundles with basic H-flux.

More generally,
as argued in \cite{BHM3}, when the $\TT^n$-action on the principal
torus bundle $Y$ does not lift to an $\RR^n$-action on
$\alg$, one has to leave the category of $C^*$-algebras
in order to be able to define a ``twisted" lift. The associator in this case is the restriction
of the H-flux $H$ to the torus fibre of $Y$, and the ``twisted" crossed product is defined
to be the T-dual. The fibres of the T-dual are noncommutative,  nonassociative tori.
That this is a proper definition of T-duality is due to our results which show that
the analogs of Takai duality and the Connes-Thom isomorphism hold in this context.
Thus an appropriate context to describe nonassociative algebras that arise as T-duals
of spacetimes with background flux, such as nonassociative tori, is that of C$^*$-algebras
in tensor categories.

In the following, $\CT(X, H)$ denotes a continuous trace algebra over $X$ with Dixmier-Douady class $H\in H^3(X, \ZZ)$
cf. \cite{RW}.

\begin{theorem}\label{thm:Tduality} Let  $H_j \in H^j(X; H^{3-j}(T;\integer)) \subseteq H^3(X\times
T; \integer)$ for $j=0,1,2,3$. In the notation above,
$(X\times T, H_0+ H_1 + H_2 +H_3)$ and the parametrised strict
deformation quantization of $(Y, q^*(H_3))$ with deformation
parameter $\phi_1,\, [\phi_1]=H_1$, further deformed
nonassociatively with deformation 3-cocycle $\phi, \ [\phi] =
H_0$, are T-dual pairs, where the 1st Chern class $c_1(Y) =H_2$, that is, 
$$
(( \CT(Y, q^*(H_3))_{\phi_1})_\phi \cong \CT(X\times T, H_0 +H_1 + H_2 +
H_3)  \rtimes V,
  $$
  where $V$ is the universal cover of $T$, which acts on the algebra above cf. \cite{BHM3}.
\end{theorem}

\begin{proof}
By \cite{HM09, HM10} we know that
$$
 \CT(Y, q^*(H_3))_{\phi_1}\cong \CT(X\times T, H_1 + H_2 + H_3) \rtimes V.
 $$
It follows from the basic theory of continuous trace $C^*$-algebras that, \cite{RW}, 
\begin{align*}
\CT(X\times T, H_0 +H_1 + H_2 + H_3) &\cong \CT(X\times T,
H_1 + H_2 + H_3) \otimes_{C_0(X \times T)} \CT(X \times T, H_0)\\
& \cong \CT(X\times T,
H_1 + H_2 + H_3) \otimes_{C(T)} \CT(T, H_0)
\end{align*} 
since $\CT(X \times T, H_0) \cong C_0(X)  \otimes \CT(T, H_0)$.
Also by \cite{BHM3,
BHM4}, we have $ \CT(T, H_0)  \rtimes V \cong
A_\phi $, where we recall that $ \CT(T, H_0)
\rtimes V$ is a twisted (nonassociative) crossed product and $A_\phi$ denotes the nonassociative torus. Also,
by \S\ref{sect:class}
$$
 \CT(Y, q^*(H_3))_{\phi_1} \otimes  A_\phi  \cong
(( \CT(Y, q^*(H_3))_{\phi_1})_\phi.
$$
Therefore
$$
(( \CT(Y, q^*(H_3))_{\phi_1})_\phi \cong \CT(X\times T, H_0 +H_1 + H_2 +
H_3)  \rtimes V,
  $$
 proving the result.
\end{proof}


\section{Nonassociative principal bundles}


In \cite{BHM3} the nonassociative bundles arose as duals of
associative algebras, but in this section they are the main
object of interest, so we shall take $G = \wh{V}$, where $V$ is
a vector group. The quotient of $V/\Lambda$ by a maximal rank
lattice $\Lambda$ is a torus $T$, whose Pontrjagin dual $\wh{T}
\cong \Lambda^\perp \subset \wh{V}$ is the reciprocal lattice
in the dual group which is trivial on $\Lambda$.

We recall that an algebra $\alg(X)$ over a locally compact
Hausdorff space $X$ is an algebra with a homomorphism from
$C_0(X)$ to the centre of the multiplier algebra ${\mathcal
M}\alg(X)$, and, for consistency in the nonassociative case, we
require this action to commute with that of $C_0(T)$.

This suggests the following definition:

\begin{definition} A nonassociative algebra $\alg(X)$ over $X$,
with nonassociativity defined by $\phi$,  is called a {\it
nonassociative principal $(T,\phi)$ bundle} (or {\it
$NAP(T,\phi)$-bundle}), if there is an action $\gamma$ of $T$
as automorphisms of $\alg(X)$ (commuting with the
$C_0(X)$-action) and an isomorphism of nonassociative algebras
$$
\alg(X)\rtimes_\gamma T \cong C_0(X,\cpt_\phi),
$$
for twisted compact operators on some space.
\end{definition}

\bigskip
\begin{example} This is clearly modelled on the definition of a
noncommutative principal bundle in \cite{ENOO}, to which it
reduces in the associative case when $\phi \equiv 1$. As a
concrete example we may consider the nonassociative torus
algebra generated by unitary elements $U(\xi)$ for $\xi \in
\wh{T}$ satisfying
$$
U(\xi)(U(\eta)U(\zeta)) = \phi(\xi,\eta,\zeta)(U(\xi)U(\eta))U(\zeta)
$$
for all $\xi, \eta, \zeta \in \wh{T}$. There is an obvious
action of $v\in T$ action given by $v:U(\xi) \mapsto
\xi(v)U(\xi)$ which preserves the above defining relation.
\end{example}


The definition is also motivated in part by the following
observation. Recalling that $\wh{T} \cong \Lambda^\perp \subset
\wh{V}$, \cite[Theorem 8.3]{BHM3} tells us that, in the case
discussed there, the nonassociative torus bundle is isomorphic
to
$$
C_0(X,\cpt_\phi(L^2(\wh{V}))\rtimes_{\gamma,u}\Lambda^\perp
 = C_0(X,\cpt_\phi(L^2(\wh{V}))\rtimes_{\gamma,u}\wh{T},
$$
where the crossed product is a Leptin--Busby-Smith crossed
product (\cite[Section 3]{BHM3}) defined by an action $\gamma$
and an algebra valued cocycle $u$ satisfying
\begin{eqnarray*}
\gamma_x\circ\gamma_y &=& \ad(u(x,y))\circ\gamma_{xy},\\ \phi(x,y,z)
u(x,y)u(xy,z) &=& \gamma_x[u(y,z)]u(x,yz).
\end{eqnarray*}
 On the other hand, such a crossed product has a dual action
 $\wh{\gamma}$ of $T$ and if we knew a suitable version of Takai duality \cite[Theorem
 9.2]{BHM3} for nonassociative algebras we should have
$$
C_0(X,\cpt_\phi(L^2(\wh{V}))\rtimes_{\gamma,u}\wh{T}\rtimes_{\wh{\gamma}}T
\cong
C_0(X,\cpt_\phi(L^2(\wh{V}))\otimes \cpt_{\overline{\phi}}(L^2(\wh{V}),
$$
showing that the definition does match what happens in the
known nonassociative case, apart, possibly, from the
nonassociativity.

To investigate this last point we shall now  show that the
multiplier $u$ does not compound the existing nonassociativity.

\bigskip
\begin{proposition} The action $\gamma$ of $\wh{T} = \Lambda^\perp\subset \wh{V}$
on $C_0(X,\cpt_\phi(L^2(\wh{V})))$ defined (with the $X$ argument suppressed)
by $\gamma_\xi[K](\eta,\zeta) =
\phi(\xi,\eta,\zeta)^{-1}K(\xi^{-1}\eta,\xi^{-1}\zeta)$ has a trivial
associativity cocycle.
\end{proposition}

\bigskip\noindent{\it Proof.} We calculate that
\begin{eqnarray*}
(\gamma_\omega\gamma_\xi[K])(\eta,\zeta) &=&
\phi(\omega,\eta,\zeta)^{-1}\phi(\xi,\omega^{-1}\eta,\omega^{-1}\zeta)^{-1}
K(\xi^{-1}\omega^{-1}\eta,\xi^{-1}\omega^{-1}\zeta)\\ &=&
\frac{\phi(\omega\xi,\eta,\zeta)}{\phi(\omega,\eta,\zeta)\phi(\xi,\omega^{-1}\eta,\omega^{-1}\zeta)}
\gamma_{\omega\xi}[K](\eta,\zeta)\\ &=&
\phi(\xi,\omega,\zeta)\phi(\xi,\eta,\omega)\gamma_{\omega\xi}[K](\eta,\zeta)\\
&=&
\phi(\omega,\xi,\eta)\phi(\omega,\xi,\zeta)^{-1}\gamma_{\omega\xi}[K](\eta,\zeta)\\
&=& \ad(u(\omega,\xi))\gamma_{\omega\xi}[K](\eta,\zeta),
\end{eqnarray*}
where $u(\omega,\xi)$ is the multiplication operator
$\psi(\eta) \mapsto \phi(\omega,\xi,\eta)\psi(\eta)$ on
$L^2(\wh{V})$. Now the associativity cocycle is given by
\begin{eqnarray*}
u(\xi,\eta)u(\xi\eta,\zeta)u(\xi,\eta\zeta)^{-1}\gamma_\xi[u(\eta,\zeta)]^{-1}
&=& \phi(\xi,\eta,\cdot)\phi(\xi\eta,\zeta,\cdot)
\phi(\xi,\eta\zeta\cdot)^{-1}\gamma_\xi[\phi(\eta,\zeta,\cdot)]^{-1}\\ &=&
\phi(\xi,\eta,\cdot)\phi(\xi\eta,\zeta,\cdot)
\phi(\xi,\eta\zeta\cdot)^{-1}\phi(\eta,\zeta,\xi^{-1}\cdot)^{-1}\\ &=&
\phi(\eta,\zeta,\xi),
\end{eqnarray*}
and for $\xi$, $\eta$, $\zeta \in \Lambda^\perp$, we have $\phi(\eta,\zeta,\xi)
= 1$. \quod

This result shows that in \cite[Theorems 8.2, 8.3]{BHM3} the twisted crossed
product by $\Lambda^\perp \cong \wh{T}$ does not introduce any extra
nonassociativity into
$$
\alg(X) \cong C_0(X,\cpt_\phi(L^2(\wh{V})) \rtimes_{\gamma,u}\Lambda^\perp.
$$


\section{Strictification and duality}\label{sect:strictify}


In working directly from our definition our first task is to
get a reverse Takai duality in which we start with $\alpha \sim
\wh{\gamma}$ and then do a twisted product so that we can get
$$
\alg(X) \cong C_0(X,\cpt_\phi(L^2(\wh{V}))\rtimes_{\gamma,u}\wh{T}.
$$
We shall approach this by proving a new, particularly
interesting duality Theorem which arises from combining the
duality Theorem \cite[Theorem 9.2]{BHM3} with the
strictification procedure in \cite[Appendix A]{BHM4}.
In this section, we will take $G$ to be an abelian Lie group.

Before stating the Theorem we recall that, according to
MacLane's strictification Theorem, any monoidal category is
equivalent to a strict monoidal category, in which the
associativity map $A\otimes (B\otimes C) \to (A\otimes
B)\otimes C$ is the obvious identification by rebracketing:
$a\otimes (b\otimes c) \mapsto (a\otimes b)\otimes c$. In
\cite{BHM4} it was shown that, when the objects are
$C_0(G)$-modules (or $C^*(\wh{G})$-modules) and the associator
map is given by the action of $\phi \in {\mathcal
UM}(C_0(G)\times C_0(G)\times C_0(G))$, the equivalence is
particularly simple, and is given by a functor taking an object
$A$ to $A\otimes C_0(G)$. This could be identified with the
crossed product $A \rtimes \wh{G}$ and it is useful to give its
Fourier transformed version. In the following theorem , let $\alg$ be 
such an algebra in a monoidal category as in \cite{BHM4},
with associator $\phi$ as above.

\begin{theorem} Let $\alg$ be an algebra in a monoidal category as 
above, with associator $\phi$.
The multiplication for the crossed product $\alg\rtimes\wh{G}$
can be written in terms of $\wh{a}(\xi) = \int a(x)\xi(x)\,dx$
as
$$
(a\star b)(x) = a_x[b(x)].
$$
where the action of $a_x \in \alg\otimes C_0(G)$ is a
combination of the multiplication in $\alg$ and the action of
$C_0(G)$.
\end{theorem}

\begin{proof} We shall write the composition in the abelian group $G$ additively, but that in $\widehat{G}$ multiplicatively.
\begin{eqnarray*}
\wh{(a\star b})(\xi) &=& \int
\wh{a}(\eta)\alpha_\eta[\wh{b}(\xi\eta^{-1})]\,d\eta\\ &=&\int
a(y)\eta(y)\alpha_\eta[b(x)](\xi\eta^{-1})(x)\,dxdyd\eta\\ &=&\int
a(z+x)\eta(z)\alpha_\eta[b(x)]\xi(x)\,dxdzd\eta\\ &=&\int
a_x(z)\eta(z)\alpha_\eta[b(x)]\xi(x)\,dxdzd\eta
\end{eqnarray*}
where $a_x(z) = a(z+x)$. Now
$$
 (f.b)(x) = \int f(z)\eta(z)\alpha_\eta[b(x)]\,dz d\eta
$$
links the action of $f\in C_0(G)$ with that of $\wh{G}$, so
$$
\wh{(a\star b})(\xi) =\int a_x(z)\eta(z)\alpha_\eta[b(x)]\xi(x)\,dxdzd\eta
=\int a_x[b(x)]\xi(x)\,dx,
$$
showing that
$$
\hfill(a\star b)(x) = a_x[b(x)]. \hspace{1in}
$$
\end{proof}

This can be identified with the crossed product $A \rtimes
\wh{G}$, but we wish to modify the tensor product
$\otimes_{C_0(G)}$ to $\circ$ so that
$$
(A\circ B)\otimes_{C_0(G)} C \equiv A\otimes_{C_0(G)}
(B\otimes_{C_0(G)} C).
$$
(An explicit construction shows that this is possible: we
decompose $\phi^{-1} \in {\mathcal UM}(C_0(G))\otimes {\mathcal
UM}(C_0(G))\otimes {\mathcal UM}(C_0(G))$ as
$\phi^\prime\otimes\phi^{\prime\prime}\otimes
\phi^{\prime\prime\prime}$ and then act by $a\otimes b \mapsto
[(\phi^\prime.a)\otimes(\phi^{\prime\prime}.b)].\phi^{\prime\prime\prime}$.)
This enables us to get a functor sending $\alg\circ\alg \to
\alg\otimes \alg$.

\begin{definition} Changing the multiplication $\alg\otimes\alg \to \alg$
 an algebra $\alg$ to the multiplication
$\alg\circ\alg \to \alg\otimes \alg \to \alg$ yields the
modified crossed product $\alg \rtimes^\phi \wh{G}$.
\end{definition}

\bigskip
(We note that this crossed product is defined for any
antisymmetric tricharacter $\psi$, and not just for $\phi$. In
general it changes the nonassociativity from that defined for
one tricharacter to that defined by another.)

\begin{theorem} Let $\bdd$ be a $C^*$-algebra and $\alg = \bdd\rtimes_{\beta,u}G$ be a
nonassociative Leptin--Busby--Smith generalised crossed
product, with associator defined by $\phi$. Then for the dual
action $\wh{\beta}$ of $\wh{G}$, we have
$$
\alg\rtimes^{\psi\phi}_{\wh{\beta}} \wh{G} \cong \bdd \otimes
\cpt_{\psi}(L^2(G)),
$$
and, in particular,
$$
\alg\rtimes^{\phi}_{\wh{\beta}} \wh{G} \cong \bdd \otimes \cpt(L^2(G)),
$$
and
$$
\alg\rtimes_{\wh{\beta}} \wh{G} \cong \bdd \otimes \cpt_{\overline{\phi}}(L^2(G)).
$$
An ordinary crossed product $\alg\rtimes_\alpha G$ has a
natural $\wh{G}$ action $\wh{\alpha}$ and $(\alg\rtimes_\alpha
G)\rtimes_{\wh{\alpha}}\wh{G}$ with the double dual action
$\wh{\wh{\alpha}}$ of $\wh{\wh{G}} \cong G$, is isomorphic to
$\alg\otimes\cpt(L^2(G))$ with the action $\alpha$ on the first
factor and the adjoint action of the right regular
representation of $G$ on the compact operators. This is true
for associative and nonassociative algebras.
\end{theorem}

\bigskip\noindent{\it Proof.} We shall just give the
algebraic form of the isomorphism. The technical details
justifying the formal calculations involve showing that
subspaces of compactly supported continuous functions in the
domain and range of the isomorphism are dense, and those
details are the same as in the usual case. The associativity of
$\alg$ is not relevant to the argument, which can even be
generalised to the context of monoidal categories.

The first special case can be proved by noting that
\cite[Theorem 9.2]{BHM3} gives an isomorphism
$$
\alg\times^\phi_{\wh{\beta}} \wh{G} \cong \bdd \otimes \cpt_{\overline{\phi}}(L^2(G)),
$$
and by \cite[Example A.1]{BHM4} the strictification sends
$\cpt_{\overline{\phi}}(L^2(G))$ to $\cpt(L^2(G))$, so that
the result follows.

In the general case we give the direct construction. We first
note that for $\alg = \bdd\rtimes G$ the elements of $\alg$ are
functions on $G$, with $C_0(G)$ acting by multiplication, and
elements of the crossed product can be regarded as functions on
$G\times G$, with multiplication
\begin{eqnarray*}
(a\star b)(s,x) &=&
(\phi^\prime.a)_x[(\phi^{\prime\prime}.b)(x)].\phi^{\prime\prime\prime}(x)\\
&=& \int
\phi^\prime(t)a_x(t,\cdot)\beta_t[\phi^{\prime\prime}b(s-t,x)].\phi^{\prime\prime\prime}(x)u(t,s-t)\,dt\\
&=& \int \phi(t,s-t,x)a_x(t,s-t)\beta_t[b(s-t,x)]u(t,s-t)\,dt.
\end{eqnarray*}
We now introduce the transform
$$
a(t,x) = \beta_{t+x}[\wt{a}(t+x,x)]u(t,x)^{-1}
$$
with inverse $\wt{a}(w,z) = \beta_w^{-1}[a(w-z,z)u(w-z,z)]$.
Then, recalling the relation between $u$ and $\phi$ and using
the antisymmetry properties of $\psi$,
\begin{eqnarray*}
\beta_w[(\wt{a\star b})(w,z)] &=& (a\star
b)(w-z,z)u(w-z,z)\\
 &=& \int (\psi\phi)(t,w-z-t,z)a_z(t,w-z-t)\\&\qquad&\qquad\beta_t[b(w-z-t,z)]u(t,w-z-t)u(w-z,z)\,dt\\
 &=& \int \psi(t,w-z-t,z)a(t,w-t)\beta_t[b(w-z-t,z)u(w-z-t,z)]u(t,w)\,dt\\
 &=& \int \psi(t,w-t,z)\beta_w[\wt{a}(w,w-t)]u(t,w)^{-1}\beta_t[\beta_{w-t}[\wh{b}(w-z-t,z)]]u(t,w)\,dt\\
 &=& \int \psi(t,w,z)\beta_w[\wt{a}(w,w-t)]\beta_w[b(w-t,z)]\,dt\\
 &=& \int \psi(t,w,z)\beta_w[\wt{a}(w,v)]\beta_w[b(v,z)]\,dv
\end{eqnarray*}
showing that we have $\psi$-twisted multiplication of compact
kernels. When $\psi \equiv 1$ we get the normal untwisted
multiplication on the right, and when $\psi = \phi^{-1}$ we get
an ordinary crossed product on the left. \quod

This result does not depend on associativity of the original
algebra, and enables us to confirm that the nonassociative
bundles of \cite{BHM3} are $NAP(T,\phi)$-bundles.

\begin{corollary}
For $\alg =
C_0(X,\cpt_\phi(L^2(\wh{V})))\rtimes_{\gamma,u}\wh{T}$ there is
an isomorphism $$ \alg\rtimes_{\wh{\gamma}} T \cong
C_0(X,\cpt_\phi(L^2(\wh{V})))\otimes \cpt(L^2(\wh{T}))).
$$
\end{corollary}

\bigskip\noindent{\it Proof.}
With $G = \wh{T}$ and  $\psi=\phi^{-1}$, we have
$$
\alg\times_{\wh{\gamma}}T \cong C_0(X,\cpt_\phi(L^2(\wh{V})))\otimes \cpt_{\overline{\phi}_\Lambda}(L^2(\wh{T})),
$$
where $\phi_\Lambda$ is the restriction of $\phi$ to
$\Lambda^\perp\times\Lambda^\perp\times\Lambda^\perp$ which we
showed in the last section was trivial. Thus the result
follows. \quod

\begin{corollary} For a torus $T = V/\Lambda$ a nonassociative principal
$(T,\phi)$ bundle $\alg(X)$ satisfies
$$
\alg(X)\otimes\cpt(L^2(\wh{V})) \cong
C_0(X,\cpt_\phi(L^2(\wh{V})))\rtimes_{\wh{\alpha}} \Lambda^\perp.
$$
\end{corollary}

\bigskip\noindent{\it Proof.} With the appropriate change of notation (and using
$\wh{\alpha}$ for both the dual action and the equivalent
action on kernels) our generalised Takai duality gives
$$
\alg(X)\otimes\cpt(L^2(T)) \cong
(\alg(X)\rtimes_\alpha T)\rtimes_{\wh{\alpha}} \wh{T} \cong
C_0(X,\cpt_\phi(L^2(\wh{V})))\rtimes_{\wh{\alpha}} \Lambda^\perp,
$$
as asserted.

In this last case the double dual action of $G$ is given by
$\wh{\wh{\alpha}}_v[F](x,\xi) = \xi(v)F(x,\xi)$, which Fourier
transforms to
\begin{eqnarray*}
\wt{\wh{\wh{\alpha}}_v[F]}(x,z) &=& \int
\xi(v)\alpha_x^{-1}[F(xz^{-1},\xi)]\xi(z)\,d\xi\\ &=& \alpha_v[\int
\alpha_{xv}^{-1}[F((xv)(zv)^{-1},\xi)]\xi(zv)\,d\xi]\\
&=&\alpha_v[\wt{F}(xv,zv)],
\end{eqnarray*}
which combines the action of $\alpha_v$ on $\alg$ with the
adjoint action of the right regular representation on kernels.
\quod

This Corollary shows that the result of \cite[Theorem
8.3]{BHM3} is valid for general $NAP(T,\phi)$-bundles, and not
just for the dual bundles considered there.

\bigskip
The strictifiction duality result would enable us to give a
more general notion of a nonassociative principle $T$-bundle by
asking that there is a $T$-action such that the strictified
crossed product with $T$ is $C(X,\cpt(L^2(T))$. However, this
is less useful than at first appears because the
nonassociativity factor is obvious from the algebra and, in any
case, is not lifted from the torus $T$.


\section{Classification of nonassociative principal torus bundles via nonassociative strict deformation quantization}
\label{sect:class}


The next step will be to classify $NAP(T,\phi)$-bundles more generally for a given $\phi$. At this point we shall see that they automatically include the noncommutative case of \cite{ENOO, BHM3}, as well as the geometric case of principal bundles.

\begin{theorem} For a given $\phi \in Z^3(\widehat T)$, each NAP$(T,\phi)$-bundle is associated to an element $\sigma \in C(X,Z^2(\wh{T}))$ and a principal $T$-bundle $E$ over $X$. Conversely these data can be used to construct an associated NAP$(T,\phi)$-bundle.
\end{theorem}

\bigskip\noindent{\it Proof.}
When $\phi \equiv 1$ this is basically the same as Theorem 2.2 of \cite{ENOO} (without the technical conditions such as compactly generated and second countability which are automatic for tori and their duals), and much of the proof simply follows the same lines, so we shall
concentrate on the new aspects introduced by the nonassociativity.

The definition of an $NAP(T,\phi)$-bundle includes $\phi$ and the fact that there is an action, $\alpha$, of $T$. The Corollary tells us that $\alg(X)$ is stably equivalent to $C_0(X,\cpt_\phi(L^2(\wh{V}))\rtimes_{\gamma} \Lambda^\perp$, where $\gamma_\lambda = \wh{\alpha}$ is the dual action of $\Lambda^\perp \cong \wh{T}$ action. The action of $\lambda \in \Lambda^\perp$ on the twisted compact operators is given by $$
\gamma_\lambda[K](\xi,\eta) = \phi(\lambda,\xi,\eta)^{-1}K(\lambda^{-1}\xi,\lambda^{-1}\eta)
$$
and has a multiplier $u(\lambda,\mu)(\xi) = \phi(\lambda,\mu,\xi)$, \cite{BHM3}. The cocycle $u$ satisfies
\begin{eqnarray*}
\gamma_\lambda\circ\gamma_\mu &=&
\ad(u(\lambda,\mu))\circ\gamma_{\lambda\mu},\\ \phi(\lambda,\mu,\nu)
u(\lambda,\mu)u(\lambda\mu,\nu) &=&
\gamma_\lambda[u(\mu,\nu)]u(\lambda,\mu\nu).
\end{eqnarray*}
Other multipliers $u$ can satisfy these relations, but the freedom in the choice is very limited since, in the first condition, two choices $u_1$ and $u_2$ will give
$$
\ad(u_1(\lambda,\mu))\circ\gamma_{\lambda\mu}
= \gamma_\lambda\circ\gamma_\mu = \ad(u_2(\lambda,\mu))\circ\gamma_{\lambda\mu},
$$
forcing $\sigma(\lambda,\mu) = u_1(\lambda,\mu)u_2(\lambda,\mu)^{-1}$ to be central, that is in $C_b(X)$, and so also $\gamma$-invariant, since, by definition, $\gamma$ commutes with the action of  $C_b(X)$.  The centrality and invariance of $\sigma$ allow us to make a comparison
of the other conditions $\phi(\lambda,\mu,\nu) u_j(\lambda,\mu)u_j(\lambda\mu,\nu) =
\gamma_\lambda[u_j(\mu,\nu)]u_j(\lambda,\mu\nu)$, for $j = 1,2$, which tells us that $\sigma$ is an ordinary $C_b(X)$-valued cocycle: $\sigma(\lambda,\mu)\sigma(\lambda\mu,\nu) =
\sigma(\mu,\nu)\sigma(\lambda,\mu\nu)$. Moreover, since $u_1$ and $u_2$ are unitary so is $\sigma$. We then argue as in \cite[Section 2]{ENOO} or \cite{HM10} to see that the only remaining freedom is given by $H^1(X,\wh{T})$, which is the choice of an ordinary principal
$T$-bundle $E$ over $X$. One can also argue as in \cite[Theorem 7.2]{ENOO} that $\sigma$ homotopic to a constant gives rise to K-trivial bundles. 

Conversely, suppose that we are given the data as above, $(\phi, \sigma, T\to E\to X)$. Consider the abelian $T-C^*$-algebra of continuous 
functions on $E$ vanishing at infinity, $C_0(E)$. By the construction in \cite{HM09, HM10}, we can parametrize strict deform quantize this $T-C^*$-algebra using
$\sigma$ as deformation parameter, to get a
new $T-C^*$-algebra denoted by $C_0(E)_\sigma$, which is a noncommutative principal torus bundle. 
By the construction in section 1, we can deform the noncommutative principal torus bundle  $C_0(E)_\sigma$ by $\phi$
to get the desired $NAP(T,\phi)$-bundle  $\left(C_0(E)_\sigma\right)_\phi$.
\quod

\appendix
\section{Octonions as a nonassociative strict deformation quantization
and nonassociative principal bundle}

The nonassociative real algebra of octonions can also be regarded as a
nonassociative principal $\ZZ_2^3$-bundle over a one point space $X$. As noted
in \cite[pp 669-670]{BHM3,BHM4,B,M}, it is known that the octonions ${\mathbb
O}$ can be described as a twisted group algebra of $\wh{T} = \ZZ_2^3$, which is
generated by $\{e({\bf a}): {\bf a} \in \ZZ_2^3\}$ so that
$$ e({\bf a})e({\bf b}) = u({\bf a},{\bf b})e({\bf a}+{\bf b}),
$$ where
$$ u({\bf a},{\bf b}) = (-1)^{\sum_{i<j} a_ib_j + a_1a_2b_3 + a_3a_1b_2 + a_2a_3b_1}.
$$
Since everything is real, this realisation of the octonions can also be
identified with a real Leptin--Busby--Smith crossed product of $\RR$ by $\wh{T}
= \ZZ_2^3$, acting as trivial automorphisms $\beta({\bf a})$, and
algebra-valued multiplier $u$, whose adjoint action is trivial and so
consistent with $\beta$. The nonassociativity is evident since
\begin{eqnarray*}
(e({\bf a})e({\bf b}))e({\bf c}) &=& u({\bf a},{\bf b})e({\bf a}+{\bf b})e({\bf c})\\
&=& u({\bf a},{\bf b})u({\bf a}+{\bf b},{\bf c})e({\bf a}+{\bf b}+{\bf c})\\
&=& \frac{u({\bf a},{\bf b})u({\bf a}+{\bf b},{\bf c})} {u({\bf a},{\bf b}+{\bf
c})u({\bf b},{\bf c})} e({\bf a})(e({\bf b})e({\bf c})),
\end{eqnarray*}
and the factor $\phi$ which measures the nonassociativity can be calculated to
be
$$
\phi({\bf a},{\bf b},{\bf c}) = \frac{u({\bf a},{\bf b}+{\bf c})u({\bf b},{\bf
c})} {u({\bf a},{\bf b})u({\bf a}+{\bf b},{\bf c})} = (-1)^{{\bf a}.({\bf
b}\times{\bf c})}.
$$
This is a finite group version of Example 3.2 with $\{U(\xi):\xi\in \wh{T}\}$ replaced by $\{e({\bf
a}):{\bf a}\in \ZZ_2^3\}$, and there is an action of $\wh{\beta}$ of ${\bf p}
\in T = \wh{\ZZ_2^3} \cong \ZZ_2^3$ by
$$
\wh{\beta}({\bf p})[e({\bf a})] = (-1)^{{\bf p}\cdot{\bf a}} e({\bf a}).
$$
Since this action of the dual group involves only real-valued characters, and
the constructions in \cite[Theorem 9.2]{BHM3} are independent of the field, we
can produce a real version of the arguments leading to Theorem 4.3, and
construct the twisted crossed product ${\mathbb O}\rtimes^\phi \ZZ_2^3$ which,
by the real version of Theorem 4.3, is $\RR\otimes \cpt(\ell^2(\ZZ_2^3))$.
Since $\cpt(\ell^2(\ZZ_2^3))$ can be identified withe the matrix algebra
$M_8(\RR)$. Theorem 4.3 also tells us that the crossed product ${\mathbb
O}\rtimes^\phi \ZZ_2^3$ is the $\phi$-nonassociative matrix algebra. (See also
the comment at the end of \cite[Appendix 1]{BHM4}.) This means that the
octonions are an $NAP(\ZZ_2^3,\phi)$-bundle over a point.

\end{document}